\documentclass[12pt,a4paper,draft]{article}
\usepackage{amsmath,amsfonts,amssymb,amscd,amsthm}
\usepackage{latexsym}
\usepackage[T2A]{fontenc}
\usepackage[cp1251]{inputenc}
\usepackage{graphicx}

\theoremstyle{plain} 
\newtheorem{theorem}{Theorem}[section]
\newtheorem{lemma}[theorem]{Lemma}
\newtheorem{corollary}[theorem]{Corollary}

\newtheorem{proposition}[theorem]{Proposition}

\theoremstyle{definition} 
\newtheorem{definition}[theorem]{Definition}
\theoremstyle{remark} 
\newtheorem{remark}[theorem]{Remark}

\newcommand{\bbC}{\mathbb{C}}
\newcommand{\bbN}{\mathbb{N}}

\newcommand{\bfm}{\mathbf{m}}

\newcommand{\mcP}{\mathcal{P}}

\begin{document}
\title{Algebraic independence of multipliers of periodic orbits in the space of polynomial maps of one variable}
\author{Igors Gorbovickis}
\maketitle

\begin{abstract}
We consider a space of complex polynomials of degree $n\ge 3$ with $n-1$ distinguished periodic orbits. We prove that the multipliers of these periodic orbits considered as algebraic functions on that space, are algebraically independent over~$\bbC$.
\end{abstract}

\section{Introduction}
Consider a monic polynomial of degree $n\ge 3$ and its $n-1$ distinct periodic orbits. 
If these periodic orbits are of multiplicity $1$, then by means of the implicit function theorem their multipliers can be analytically continued to $n-1$ algebraic functions defined on some ramified cover over the space of monic polynomials of degree $n$. In this paper we show that these $n-1$ functions are algebraically independent over $\bbC$. In other words, they do not satisfy any polynomial relation with complex coefficients. In fact, we prove a slightly more general statement. Precise definitions follow.

The above mentioned result in a certain way extends the result of Yu.~Zarhin~\cite{Zarhin_1}, and answers the question formulated in Remark~1.10 of the same paper. In~\cite{Zarhin_1} the independence of multipliers is proved with a strong restriction on the \emph{periods} of the corresponding orbits. More precisely, in~\cite{Zarhin_1} it is shown that if $m_1,\dots,m_k$ are periods of some $k$ distinct periodic orbits of the polynomial $p_0(z)=z^n$, where $n\ge 3$, and the following two conditions are satisfied:

(i) $m_1+\dots+m_k\le n$,

(ii) if $m_j=1$ for some $j$, then additionally $m_1+\dots+m_k\le n-1$,

\noindent then the multipliers of those $k$ periodic orbits considered as locally analytic functions on the space of monic polynomials of degree $n$, have linearly independent differentials. Since the multipliers are (multiple valued) algebraic functions of the polynomial, this immediately implies their algebraic independence over $\bbC$ provided that conditions~(i) and~(ii) are satisfied. 



On the other hand, in~\cite{Zarhin_1}~(Remark~1.10) it is shown that even for two periodic orbits of $p_0(z)=z^n$, if conditions (i) and~(ii) are not satisfied, then it is possible for the two multipliers to have linearly dependent differentials at the polynomial $p_0(z)$. It is also questioned whether in such cases the two multipliers are still algebraically independent? Since we prove independence of multipliers with no restrictions on the periods of the corresponding orbits, our result gives a positive answer to this question. 

We prove our main result in the following way: as it was mentioned in the beginning, the choice of $n-1$ distinct non-multiple periodic orbits of a polynomial defines an irreducible algebraic set that is a ramified cover over the space of monic polynomials of degree $n$. The multipliers that we consider, are algebraic functions on that algebraic set. We show that at some point on that set the differentials of the multipliers are linearly independent, which implies the desired algebraic independence of the multipliers. 
We also find it intriguing, whether our result can be obtained using A.~Epstein's transversality theory~\cite{Epstein_1}.



\subsection{The space of polynomials with $k$ distinguished periodic orbits}
It is easy to see that even though the change of coordinates $z\mapsto z+c$ in the domain of the polynomial, modifies the polynomial, it only shifts all periodic points by $c$ and does not change their multipliers. Thus we will consider two polynomials to be equivalent, if there exists a change of coordinates of the form $z\mapsto z+c$ that brings one polynomial to another. The factor space of all monic polynomials of degree $n$ factored by this equivalence relation can be identified with so called centered polynomials.

\begin{definition}
By $\mcP^n\subset\bbC[z]$ we denote the set of centered monic polynomials of degree $n$, which are monic polynomials of degree $n$ with the term of degree $n-1$ being equal to zero.
\end{definition}
\begin{definition}
We say that a periodic orbit is of period $m$, if this periodic orbit consists of $m$ \emph{distinct} points. A point of period $m$ is a point that belongs to a periodic orbit of period $m$.
\end{definition}

Consider a polynomial $p\in \mcP^n$ and its $k$ non-multiple periodic points $z_1, \dots, z_k$ belonging to \textit{different} periodic orbits of (minimal) periods $m_1,\dots, m_k$ respectively. By $\bfm$ denote the vector of periods $\bfm=(m_1,\dots,m_k)$. With any such polynomial and its periodic points belonging to different periodic orbits, one can associate the set $M^n_\bfm$ defined in the following way:
\begin{definition}\label{Mnm_def}
The set $M^n_\bfm=M^n_\bfm(p,z_1,\dots,z_k)$ is the maximal irreducible analytic subset of $\mcP^n\times\bbC^k$, such that

(i) $(p,z_1,\dots,z_k)\in M^n_\bfm$;

(ii) For $(q,w_1,\dots,w_k)\in M^n_\bfm$, the points $w_1,\dots,w_k$ satisfy the equations $q^{\circ m_j}(w_j)=w_j$, for any $j=1,2,\dots,k$.
\end{definition}
\begin{remark}\label{irred_rem}
It immediately follows from the definition that $M^n_\bfm=M^n_\bfm(p,z_1,\dots, z_k)$ is an irreducible algebraic set. 
\end{remark}

A priori it is not obvious whether the sets $M^n_\bfm$ can be different for different initial choices of $(p, z_1,\dots, z_k)$. In Section~\ref{Perm_sec} we will prove the following lemma, which says that all these sets are the same. 
\begin{lemma}\label{Mnkm_lemma}
The set $M^n_\bfm$ is independent of the initial choice of the polynomial $p\in\mcP^n$ and its periodic points $z_1,\dots,z_k$. Thus, the set $M^n_\bfm$ is completely determined by the integer $n$, and the vector $\bfm$.
\end{lemma}
Let $\pi\colon M^n_\bfm\to \mcP^n$ be the natural projection
$$
\pi\colon (q, w_1, \dots,w_k)\mapsto q.
$$
Together with the projection $\pi$ the set $M^n_\bfm$ is a ramified cover over $\mcP^n$. 

\subsection{The multiplier map}
We define the map $\Lambda\colon M^n_\bfm\to\bbC^2$ that with every point $(p,z_1,\dots,z_k)\in M^n_\bfm$ associates the vector of multipliers of periodic points $z_1,\dots, z_k$:
$$
\Lambda\colon (p,z_1,\dots,z_k)\mapsto ((p^{\circ m_1}(z_1))', (p^{\circ m_2}(z_2))',\dots,(p^{\circ m_k}(z_k))').
$$

If we fix a system of coordinates $a=(a_0,\dots,a_{n-2})$ in $\mcP^n$, naturally related to the coefficients of the polynomials so that a vector $a=(a_0,\dots,a_{n-2})\in\bbC^{n-1}$ is identified with the polynomial $p_a(z)=z^n+a_{n-2}z^{n-2}+a_{n-3}z^{n-3}+\dots+a_1z+a_0$, then coordinates $a=(a_0,\dots,a_{n-2})$ can be viewed as local coordinates on $M^n_\bfm$ around every point $(q,z_1,\dots,z_k)\in M^n_\bfm$, which is not a point of ramification. In particular, at every such point we can consider the derivative $\frac{d\Lambda}{da}=\frac{d\Lambda(p_a,z_1(a),\dots,z_k(a))}{da}$.

Now we are ready to formulate our main theorem:
\begin{theorem}\label{main_th}
Assume $n\ge 2$ and $k\le n-1$. Then for any vector $\bfm\in\bbN^k$ and any $k$-tuple of indexes $j_1,\dots,j_k$ satisfying $0\le j_1<j_2<\dots<j_k\le n-2$, the equation 
$$
\det\left(\frac{d\Lambda}{d(a_{j_1},\dots a_{j_k})}\right)=0
$$
defines an algebraic subset of codimension $1$ in $M^n_\bfm$.
\end{theorem}
An immediate corollary from Theorem~\ref{main_th} is the following:
\begin{corollary}
For $n\ge 3$, the multipliers of any $n-1$ distinct periodic orbits considered as (multiple valued) algebraic functions on $\mcP^n$, are algebraically independent over $\bbC$.
\end{corollary}
\begin{remark}\label{det_rem}
Since according to Remark~\ref{irred_rem}, $M^n_\bfm$ is an irreducible algebraic set, in order to prove Theorem~\ref{main_th}, it is sufficient for every $k$-tuple of distinct indexes $j_1,\dots, j_k$ to find a point in $M^n_\bfm$, at which 
$$
\det\left(\frac{d\Lambda}{d(a_{j_1},\dots a_{j_k})}\right)\neq 0.
$$
\end{remark}

\section{The strategy of the proof}
The key argument is formulated in the following lemma, which will be proved in Section~\ref{Jac_sec}.
\begin{lemma}\label{nondeg_l}
Assume that $k\le n-1$. Then for any $k$-tuple of distinct indexes $j_1,\dots, j_k$ satisfying $0\le j_1<j_2<\dots<j_k\le n-2$, and for any $k$-dimensional vector of periods $\bfm=(m_1,\dots,m_k)$ there exist corresponding periodic points $z_1,\dots,z_k$ of the polynomial $p_0(z)=z^n$, such that
\begin{equation}\label{jac1}
\det\left(\frac{d\Lambda}{d(a_{j_1},\dots a_{j_k})}(p_0,z_1,\dots,z_k)\right)\neq 0.
\end{equation}
\end{lemma}

Now we can complete the proof of Theorem~\ref{main_th} modulo the auxiliary results formulated earlier.
\begin{proof}[Proof of Theorem~\ref{main_th}]
Since this Jacobian in~\ref{jac1} is non-degenerate, the periodic points $z_1,\dots,z_k$ found in Lemma~\ref{nondeg_l}, belong to different periodic orbits, hence according to Lemma~\ref{Mnkm_lemma}, the point $(p_0,z_1,\dots,z_k)$ belongs to $M^n_\bfm$. Then Theorem~\ref{main_th} immediately follows from Remark~\ref{det_rem} and the result of Lemma~\ref{nondeg_l}.
\end{proof}

\section{Computation of derivatives}
Before we proceed with the proof of Lemma~\ref{nondeg_l}, we should learn how to compute the Jacobian:

Assume that $w\in\bbC$ is a non-multiple periodic point of period $m$ for a polynomial $p_{a_0}\in\mcP^n$. Then for all $a$ in some neighborhood of $a_0$ we have a periodic point $w(a)$ of the polynomial $p_a$, obtained by analytic continuation of the periodic point $w$. Thus, for all $a$ in a neighborhood of $a_0$ we can consider the corresponding multiplier $\lambda_w(a)=(p_a^{\circ m})'(w(a))$ of the periodic point $w(a)$.

\begin{lemma}\label{comp_deriv_lemma}
Let $z_0\neq 0$ be a periodic point of period $m$ of the polynomial $p_0(z)=z^n$. Then for any index $j$ satisfying $0\le j\le n-2$,
\begin{equation}\label{diff}
\frac{d\lambda_{z_0}(0)}{da_j}= (jn^{m-1}-n^m)\sum_{i=0}^{m-1}z_0^{n^i(j-n)}.
\end{equation}
\end{lemma}

Before we give a proof of Lemma~\ref{comp_deriv_lemma}, we can establish the following corollary:
\begin{corollary}\label{poly_cor}
For every positive integers $n$, $m$ with $n\ge 2$ and every index $j$ satisfying $0\le j\le n-2$, there exists a nonzero polynomial $P_{n,j,m}(z)$, such that if $z_0\neq 0$ is a periodic point of period $m$ of the polynomial $p_0(z)=z^n$, then
$$
\frac{d\lambda_{z_0}(0)}{da_j}= \frac{1}{z_0}P_{n,j,m}\left(\frac{1}{z_0}\right).
$$
Moreover,
\begin{equation}\label{deg_eq}
\deg P_{n,j,m}=
\begin{cases}
n^m-jn^{m-1}-1, &\text{for $1\le j\le n-2$,} \\
n^{m-1}-1, &\text{for $j=0$}.
\end{cases}
\end{equation}
\end{corollary}
\begin{proof}
The existence of such polynomials immediately follows from~(\ref{diff}). When $j$ satisfies $1\le j\le n-2$, then
$$
\frac{d\lambda_{z_0}(0)}{da_j}= \frac{jn^{m-1}-n^m}{z_0}\sum_{i=0}^{m-1}\left(\frac{1}{z_0}\right)^{n^i(n-j)-1},
$$
thus we can choose $P_{n,j,m}(z)=(jn^{m-1}-n^m)\sum_{i=0}^{m-1}z^{n^i(n-j)-1}$ and then $\deg P_{n,j,m}=n^m-jn^{m-1}-1$.

When $j=0$, keeping in mind that $z_0^{n^m}=z_0$, we can rewrite~(\ref{diff}) in the following way:
$$
\frac{d\lambda_{z_0}(0)}{da_j}= \frac{-n^m}{z_0}\sum_{i=0}^{m-1}\left(\frac{1}{z_0}\right)^{n^i-1}.
$$
Then we choose $P_{n,j,m}(z)=-n^m\sum_{i=0}^{m-1}z^{n^i-1}$ and in that case $\deg P_{n,j,m}=n^{m-1}-1$.
\end{proof}

\begin{remark}\label{deriv_rem}
Notice that the $(r,s)$-th entry of the matrix in~(\ref{jac1}) is equal to the expression~(\ref{diff}) with $z_0=z_r$, $m=m_r$ and $j=j_s$. Hence, according to Corollary~\ref{poly_cor}, the $(r,s)$-th entry of the matrix in~(\ref{jac1}) can be represented as $\frac{1}{z_r}P_{n,j_s,m_r}\left(\frac{1}{z_r}\right)$.
\end{remark}
\begin{proof}[Proof of Lemma~\ref{comp_deriv_lemma}]
As before, given a vector $a=(a_0,\dots,a_{n-2})\in\bbC^{n-1}$, by $p_a(z)$ we denote the polynomial $p_a(z)=z^n+a_{n-2}z^{n-2}+a_{n-3}z^{n-3}+\dots+a_1z+a_0$.
Let $(w_0(a), w_1(a),\dots)$ be a periodic orbit of the polynomial $p_a$ with $w_i(a)=w_{i+m}(a)$ for all integers $i$ and $w_0(0)=z_0$. First we will compute the derivatives $\left.\frac{dw_i(a)}{da_j}\right|_{a=0}$. To simplify the notation, we will write $w_i=w_i(0)$.
Since $w_{i+1}(a)=p_a(w_i(a))$, we have
$$
\left.\frac{d w_{i+1}(a)}{d a_j}\right|_{a=0}= p_0'(w_i)\frac{d w_i}{d a_j}(0)+ w_i^j.
$$
Since $w_i'(a)=w_{i+m}'(a)$, from the previous equality it follows that
$$
\left.\frac{d w_{i}(a)}{d a_j}\right|_{a=0}=
\frac{\sum_{s=0}^{m-1}w_{i+s}^j\prod_{r=s+1}^{m-1}p_0'(w_{i+r})}
{1-\prod_{s=0}^{m-1}p_0'(w_{i+s})}.
$$
Now we remember that $w_{i+s}=p_0^{\circ s}(w_i)=w_i^{n^s}$ and $w_i$ is a periodic point of $p_0(z)$ of period $m$, so $w_i^{n^m}=w_i$. Since $w_0=z_0\neq 0$, we know that $w_i\neq 0$ for all $i\in\bbN$, hence $w_i^{n^m-1}=1$. Using this we get
$$
\left.\frac{d w_{i}(a)}{d a_j}\right|_{a=0}=
\frac{\sum_{s=0}^{m-1}w_i^{jn^s}\prod_{r=s+1}^{m-1}nw_i^{n^{r+1}-n^r}}{1-n^m}=
\frac{\sum_{s=0}^{m-1} n^{m-s-1} w_i^{n^m-n^{s+1}+jn^s}}{1-n^m} =
$$
$$
\frac{w_i\sum_{s=0}^{m-1} n^{m-s-1} w_i^{n^s(j-n)}}{1-n^m}=
\frac{w_0^{n^{i}}\sum_{s=0}^{m-1} n^{m-s-1} w_0^{n^{s+i}(j-n)}}{1-n^m}.
$$
Now we can compute $\frac{d \lambda_{z_0}(0)}{d a_j}$. Since
$$
\lambda_{z_0}(a)=p_a'(w_0(a))\dots p_a'(w_{m-1}(a)),
$$
by the derivative of a product formula we have
$$
\frac{d \lambda_{z_0}(0)}{d a_j}= \lambda_{z_0}(0)\sum_{i=0}^{m-1}\frac{p_0''(w_i)\left.\frac{d w_i(a)}{d a_j}\right|_{a=0}+jw_i^{j-1}}{p_0'(w_i)}=
$$
$$
n^m\sum_{i=0}^{m-1}\frac{n(n-1)w_i^{n-2}\left.\frac{d w_i(a)}{d a_j}\right|_{a=0}+jw_i^{j-1}}{nw_i^{n-1}}=
$$
$$
n^m\sum_{i=0}^{m-1}\frac{n(n-1)w_0^{n^i(n-2)}\cdot\frac{w_0^{n^i}\sum_{s=0}^{m-1} n^{m-s-1} w_0^{n^{s+i}(j-n)}}{1-n^m} +jw_0^{n^i(j-1)}} {nw_0^{n^i(n-1)}}=
$$
$$
n^{m-1}\left(\sum_{i=0}^{m-1}\frac{n(n-1)\sum_{s=0}^{m-1} n^{m-s-1} w_0^{n^{s+i}(j-n)}}{1-n^m} + jw_0^{n^i(j-n)}\right)=
$$
$$
jn^{m-1}\sum_{i=0}^{m-1}w_0^{n^i(j-n)} + \frac{n^m(n-1)}{1-n^m}\sum_{s=0}^{m-1} n^{m-s-1} \sum_{i=0}^{m-1}w_0^{n^{s+i}(j-n)}.
$$
Since $w_0^{n^{s+i}}=w_{s+i}$, and the sequence $w_0, w_1, \dots$ is periodic of period $m$, it follows that $\sum_{i=0}^{m-1}w_0^{n^{s+i}(j-n)}= \sum_{i=0}^{m-1}w_0^{n^{i}(j-n)}$. Hence
$$
\frac{d \lambda_{z_0}(0)}{d a_j}=
\left(jn^{m-1} + \frac{n^m(n-1)}{1-n^m}\sum_{s=0}^{m-1} n^{m-s-1} \right) \sum_{i=0}^{m-1}w_0^{n^i(j-n)}=
$$
$$
(jn^{m-1}-n^m)\sum_{i=0}^{m-1}z_0^{n^i(j-n)}.
$$
\end{proof}

\section{Non-degeneracy of the Jacobian}\label{Jac_sec}
In this section we give a proof of Lemma~\ref{nondeg_l}.

Let $\nu_n(m)$ denote the number of periodic points of the polynomial $p_0(z)=z^n$ with period $m$. Since this polynomial does not have multiple periodic points, the function $\nu_n(m)$ can be computed inductively by the formula
$$
n^m=\sum_{r|m}\nu_n(r),\qquad\text{or}\quad \nu_n(m)=\sum_{r|m}\mu(m/r)n^r,
$$
where the summation goes over all divisors $r\ge 1$ of $m$, and $\mu(m/r)\in\{\pm 1,0\}$ is the M\"{o}bius function.

It is easy to see from these formulas that 
\begin{equation*} 
\nu_n(m)\ge n^m-n^{m-2}, \quad\text{for }m\ge 3,\qquad\text{and}
\end{equation*}
$$
\nu_n(1)=n,\quad \nu_n(2)=n^2-n.
$$

Let $\hat\nu_n(m)$ denote the number of non-zero periodic points of the polynomial $p_0(z)=z^n$ with period $m$. Then, since zero is a fixed point of the polynomial $p_0$, it follows that $\hat\nu_n(m)=\nu_n(m)$, for $m>1$ and $\hat\nu_n(1)=\nu(1)-1$. Thus from the previous relations on $\nu_n(m)$ we obtain
\begin{equation}
\begin{array}{c}\label{nu_k3}
\hat\nu_n(m)\ge n^m-n^{m-2}, \quad\text{for }m\ge 3,\qquad\text{and} \\
\hat\nu_n(1)=n-1,\quad \hat\nu_n(2)=n^2-n.
\end{array}
\end{equation}
\begin{lemma}\label{deg_nu_lemma}
For every positive integer $n$, $m$ with $n\ge 2$ and every index $j$ satisfying $0\le j\le n-2$,
$$
\deg P_{n,j,m}<\hat\nu_n(m).
$$
\end{lemma}
\begin{proof}
The statement of the Lemma immediately follows from~(\ref{deg_eq}) and~(\ref{nu_k3}).
\end{proof}

\begin{proof}[Proof of Lemma~\ref{nondeg_l}]
We will give a proof by induction on $k$.

Case (1): $k=1$. 

If $z_1\neq 0$ is a periodic point of $p_0(z)$ of period $m_1$, then according to Corollary~\ref{poly_cor} and Remark~\ref{deriv_rem},
$$
\det\left(\frac{d\Lambda}{da_{j_1}}(p_0,z_1)\right)=
\frac{d\Lambda}{da_{j_1}}(p_0,z_1)=\frac{1}{z_1}P_{n,j_1,m_1}\left(\frac{1}{z_1}\right).
$$
Because of Lemma~\ref{deg_nu_lemma} it follows that the non-zero periodic point $z_1$ can be chosen in such a way that $P_{n,j_1,m_1}\left(\frac{1}{z_1}\right)\neq 0$, which implies that $\frac{d\Lambda}{da_{j_1}}(p_0,z_1)\neq 0$. This proves Case~(1).

Case (2): $k>1$.

Assume that $z_1,\dots, z_k\in\bbC$ are nonzero periodic points of corresponding periods $m_1,\dots,m_k$ for the polynomial $p_0(z)=z^n$. Then according to Corollary~\ref{poly_cor} and Remark~\ref{deriv_rem}, the matrix $\frac{d\Lambda}{d(a_{j_1},\dots,a_{j_k})}(p_0,z_1,\dots,z_k)$ can be written as
\begin{equation}\label{matr}
\frac{d\Lambda}{d(a_{j_1},\dots,a_{j_k})}(p_0,z_1,\dots,z_k)=
\left(
\begin{matrix}
\frac{1}{z_1}P_{n,j_1,m_1}\left(\frac{1}{z_1}\right) &\cdots & \frac{1}{z_1}P_{n,j_k,m_1}\left(\frac{1}{z_1}\right) \\
\cdots & \cdots & \cdots\\
\frac{1}{z_k}P_{n,j_1,m_k}\left(\frac{1}{z_k}\right) &\cdots & \frac{1}{z_k}P_{n,j_k,m_k}\left(\frac{1}{z_k}\right)
\end{matrix}
\right).
\end{equation}
Therefore the determinant of the matrix $\frac{d\Lambda}{d(a_{j_1},\dots,a_{j_k})}(p_0,z_1,\dots,z_k)$ can be expressed as
\begin{equation}\label{det_pol}
\det\left(\frac{d\Lambda}{d(a_{j_1},\dots,a_{j_k})}(p_0,z_1,\dots,z_k)\right) =\frac{1}{z_1}P\left(\frac{1}{z_1}\right),
\end{equation}
where $P\left(\frac{1}{z_1}\right)$ is a polynomial of $\frac{1}{z_1}$ that depends on $z_2,\dots, z_k$ as parameters. It is easy to see that 
$$
\deg P\le \max_{1\le i\le k} \deg P_{n,j_i,m_1},
$$
hence if for some choice of parameters $z_2,\dots, z_k$, polynomial $P$ is not identically zero, then according to Lemma~\ref{deg_nu_lemma}, there exists a nonzero periodic point $z_1$, for which $P\left(\frac{1}{z_1}\right)\neq 0$. Then~(\ref{det_pol}) implies that the Jacobian $\frac{d\Lambda}{d(a_{j_1},\dots,a_{j_k})}(p_0,z_1,\dots,z_k)$ is non-degenerate, which is what we need.

Now we use inductive argument to show that indeed one can choose parameters $z_2,\dots, z_k$ so that polynomial $P$ is not identically zero. Notice that the degrees of polynomials $P_{n,j_1,m_1}, \dots, P_{n,j_k,m_1}$ are all different, so for some index $i$, the degree of $P_{n,j_i,m_1}$ is maximal and greater than zero. Consider the minor of the matrix in~(\ref{matr}) obtained by deleting the first row and the $i$-th column. By inductive assumption there exist periodic points $z_2,\dots, z_k$ of corresponding periods $m_2,\dots, m_k$, such that for these periodic points the considered minor is non-degenerate. This implies that $\deg P=\deg P_{n,j_i,m_1}>0$, hence $P$ is not identically zero. This finishes the proof of Lemma~\ref{nondeg_l}.
\end{proof}

\section{Permutation of periodic orbits}\label{Perm_sec}
In This section we will give a proof of Lemma~\ref{Mnkm_lemma}. We start with two definitions:
\begin{definition}
For $n\ge 2$ we considering the set $\mcP^n_0\subset\mcP^n$ of polynomials defined in the following way: $\mcP^n_0=\{z^n+c\mid c\in\bbC\}$.
\end{definition}

\begin{definition}
For every positive integer $m$ by $X_m\subset\mcP^n$ denote the set of all polynomials $q\in\mcP^n$, such that under a small perturbation, $q$ can generate at least one new periodic orbit of period $m$. By $X_m'$ we denote the intersection $X_m'=X_m\cap\mcP^n_0$.
\end{definition}
\begin{remark}
It is not hard to show that $X_m$ is an algebraic subset of $\mcP^n$, however, for our purposes it is enough to notice that $X_m$ is contained in a codimension $1$ algebraic subset $Y_m\subset\mcP^n$ of those polynomials $q\in\mcP^n$, for which $q^{\circ m}$ has a multiple fixed point. Since $Y_m\cap\mcP^n_0$ is a finite set, $X_m'$ is also finite.
\end{remark}

\begin{proof}[Proof of Lemma~\ref{Mnkm_lemma}]
Consider a polynomial $p\in\mcP^n$ and its periodic points $z_1,\dots, z_k$ belonging to different periodic orbits of corresponding periods $m_1,\dots, m_k$. Consider the vector $\bfm=(m_1,\dots, m_k)$ and the corresponding set $M_\bfm^n=M_\bfm^n(p,z_1,\dots,z_k)$. It follows from Definition~\ref{Mnm_def} that in order to prove that the set $M_\bfm^n$ is independent of the initial choice of $(p,z_1,\dots,z_k)$, it is sufficient to show that for every polynomial $q\in\mcP^n$ and its non-multiple periodic points $w_1,\dots,w_k$ belonging to different periodic orbits of corresponding periods $m_1,\dots,m_k$, the point $(q,w_1,\dots,w_k)$ belongs to $M_\bfm^n=M_\bfm^n(p,z_1,\dots,z_k)$. In other words, it is sufficient to show that the point $(q,w_1,\dots,w_k)$ can be obtained from $(p,z_1,\dots,z_k)$ by analytic continuation of the periodic points $z_1,\dots,z_k$ along some curve $\gamma$ in $\mcP^n\setminus\cup_{j=1}^k X_{m_j}$ connecting the polynomials $p$ and $q$.

We will show that the curve $\gamma$ can be constructed from three adjacent pieces $\gamma_1$, $\gamma_2$ and $\gamma_3$. We choose $\gamma_1$ to be any curve in $\mcP^n\setminus\cup_{j=1}^k X_{m_j}$ that connects the polynomial $p$ with the polynomial $p_0(z)=z^n$. Similarly, $\gamma_3$ is any curve in $\mcP^n\setminus\cup_{j=1}^k X_{m_j}$ that connects the polynomial $p_0$ with the polynomial $q$. Analytic continuation of the periodic points $z_1,\dots,z_k$ along $\gamma_1$ produces some periodic points $x_1,\dots,x_k$ of the polynomial $p_0$, while analytic continuation of the periodic points $w_1,\dots,w_k$ along $-\gamma_3$ produces some (possibly different) periodic points $x_1',\dots,x_k'$. In order to complete the construction of the curve $\gamma$, we have to show that every point $(p_0,x_1',\dots,x_k')$ can be obtained from every point $(p_0,x_1,\dots,x_k)$ by analytic continuation along some loop $\gamma_2\subset\mcP^n\setminus\cup_{j=1}^k X_{m_j}$, where $x_1,\dots,x_k$ are periodic points of $p_0$ belonging to different periodic orbits of corresponding periods $m_1,\dots,m_k$ and similarly $x_1',\dots,x_k'$ are periodic points of $p_0$ belonging to different periodic orbits of corresponding periods $m_1,\dots,m_k$.

We will prove that the loop $\gamma_2$ can be chosen inside the set $\mcP^n_0\setminus\cup_{j=1}^k X_{m_j}'\subset\mcP^n\setminus\cup_{j=1}^k X_{m_j}$. First we notice that in the case when $m_1=m_2=\dots=m_k$, the existence of such loop $\gamma_2$ follows from the result of Schleicher~\cite{Schleicher}, which says that by choosing an appropriate element of the fundamental group $\pi_1(\mcP^n_0\setminus X_{m_1}, p_0)$ it is possible to realize any permutation of periodic orbits of $p_0$ of period $m_1$. Moreover, within each orbit of period $m_1$ any cyclic permutation of its points can be realized independently from permutations of other periodic points of the same period.

Finally, the case, when not all of $m_j$ are equal to each other, follows from the same result of Schleicher and the following proposition:

\begin{proposition}
For every two distinct positive integers $m_1\neq m_2$, the sets $X_{m_1}'$ and $X_{m_2}'$ are disjoint.
\end{proposition}
\begin{proof}
Assume the set $X_{m_1}'\cap X_{m_2}'$ is non-empty and $p\in X_{m_1}'\cap X_{m_2}'$. Then this means that polynomial $p$ has two (not necessarily distinct) periodic points $z_1$ and $z_2$ 
such that 
$$
p^{\circ m_1}(z_1)=z_1,\qquad p^{\circ m_2}(z_2)=z_2, \quad\text{and}
$$
$$
(p^{\circ m_1}(z_1))'=1,\qquad (p^{\circ m_2}(z_2))'=1.
$$
These conditions imply that the points $z_1$ and $z_2$ belong to parabolic periodic orbits of the polynomial $p$. Since the polynomial $p$ is of the form $p(z)=z^n+c$, it has only one critical point, hence, according to Corollary~10.11 in~\cite{Milnor}, it cannot have more than one parabolic periodic orbit. Thus $z_1$ and $z_2$ are from the same parabolic periodic orbit of $p$. Moreover, this periodic orbit generates new periodic orbits of two different periods $m_1$ and $m_2$ under a small perturbation of $p$ in $\mcP^n_0$. The latter is impossible.
\end{proof}

We finish the proof of Lemma~\ref{Mnkm_lemma} by splitting the periodic points $x_1,\dots,x_k$ into classes of points having the same periods. Within each class the necessary permutation of periodic points is realizable according to the above mentioned theorem of Schleicher. 
Since the sets $X_{m_j}'$ are disjoint for different values of $m_j$, this implies that the necessary permutations of periodic points can be obtained independently within each of the classes.
\end{proof}

\bibliography{mult}
\bibliographystyle{abbrv} 

\end{document}